%% file: Apost.tex
\documentclass{amsart}
\usepackage[latin1]{inputenc}
\usepackage{amsmath,bm}
\usepackage{amstext,amsbsy}
\usepackage{amsfonts}
\usepackage{amssymb}
\usepackage{color}
\usepackage{epsfig}
\usepackage{comment, enumerate}
\usepackage{xspace}
\usepackage[foot]{amsaddr}

\graphicspath{{figs/}}

\newcommand{\mesh}{{\mathcal T}}

\newcommand{\real}{{\mathbb{R}}}
\newcommand{\complex}{{\mathbb{C}}}
\renewcommand{\H}[1]{H^{#1}(\Gamma)}
\newcommand{\Hb}[1]{\bm H^{#1}(\Gamma)}
\newcommand{\Hpara}[1]{\bm H^{#1\frac12}_{\parallel}(\Gamma)}
\newcommand{\Hparas}[1]{\bm H^{#1}_{\parallel}(\Gamma)}
\newcommand{\Hortho}[1]{\bm H^{#1\frac12}_{\bot}(\Gamma)}
\newcommand{\Horthos}[1]{\bm H^{#1}_{\bot}(\Gamma)}
\newcommand{\Hdiv}{\bm H^{-\frac12}_{{\rm div}}(\Gamma)}
\newcommand{\Hdivp}[1]{\bm H^{#1}_{{\rm div}}(\Gamma)}
\newcommand{\Hcurl}{\bm H^{-\frac12}_{{\rm curl}}(\Gamma)}
\newcommand{\HcurlO}{\bm H({\curl},\Omega)}

\newcommand{\Hone}{\bm H^1(\Omega)}
\newcommand{\Lt}{\bm L_t^2(\Gamma)}
\newcommand{\lt}{L^2(\Gamma)}
\newcommand{\RT}{\bm {RT}_0}

\renewcommand{\div}{{\rm div}_\Gamma}
\newcommand{\grad}{{\bf grad}_\Gamma}
\newcommand{\gradx}{{\bf grad}_{\Gamma,\x}}
\newcommand{\curl}{\boldsymbol{\mathrm{curl}} \; }
\newcommand{\curlv}{{\bf curl}_\Gamma}
\newcommand{\curls}{{\rm curl}_\Gamma}

\newcommand{\Ak}{\bm A_k}
\newcommand{\Vk}{V_k}
\newcommand{\Gk}{G_k}
\newcommand{\Ipsi}{{\mathbf I}_{\mesh}}
\newcommand{\Ialpha}{ {\rm I}_{\mesh}}
\newcommand{\SLPs}{\Psi^V_k}
\newcommand{\SLPv}{\bm \Psi^A_k}

\newcommand{\nHdiv}[1]{\|#1\|_{\bm H^{-\frac12}_{\rm div}(\Gamma)}}
\newcommand{\nHcurl}[1]{\|#1\|_{\bm H^{-\frac12}_{\rm curl}(\Gamma)}}
\newcommand{\nHortho}[2]{\|#2\|_{\bm H^{#1\frac12}_{\bot}(\Gamma)}}

\newcommand{\nHpara}[2]{\|#2\|_{\bm H^{#1\frac12}_{\parallel}(\Gamma)}}

\newcommand{\nH}[2]{\|#2\|_{H^{#1}(\Gamma)}}
\newcommand{\nHb}[2]{\|#2\|_{\bm H^{#1}(\Gamma)}}

\newcommand{\nLt}[2]{\|#1\|_{\bm L^2({#2})}}
\newcommand{\nlt}[2]{\|#1\|_{L^2({#2})}}

\newcommand{\para}[2]{\langle #1, #2\rangle_{\parallel,\Gamma}}

\newcommand{\half}[2]{\langle  #1, #2\rangle_{\frac12,\Gamma}}
\newcommand{\threetwo}[2]{\langle  #1, #2\rangle_{\frac32,\Gamma}}

\newcommand{\ortho}[2]{\langle  #1, #2\rangle_{\bot,\Gamma}}
\newcommand{\po}[2]{~\hspace{-0.2cm}_\bot\!\langle #1, #2\rangle_{\parallel}}

\renewcommand{\u}{\bm u}
\renewcommand{\v}{\bm v}
\newcommand{\w}{\bm w}
\newcommand{\uh}{\bm U}

\newcommand{\vh}{\bm V}
\newcommand{\n}{\bm n}
\newcommand{\x}{\bm x}
\newcommand{\y}{\bm y}

\newcommand{\f}{\bm f}
\newcommand{\Psiv}{\bm \Psi}
\newcommand{\Psivm}{\bm \Psi_\mesh}
\newcommand{\alpham}{\alpha_\mesh}
\newcommand{\dPsiv}{\delta\hspace{-1.2pt}\bm \Psi}
\newcommand{\dalpha}{\delta\hspace{-1.2pt}\alpha}
\newcommand{\Res}{\bm R}
\newcommand{\cres}{r}
\newcommand{\LRes}{\Res_{\ka}}
\newcommand{\Lcres}{\cres_\ka}
\newcommand{\LResG}{\Res_0}
\newcommand{\LcresG}{\cres_0}
\newcommand{\h}{{\tt h}}

\newcommand{\calI}{\mathcal I}
\newcommand{\edges}{\mathcal E_\mesh}
\newcommand{\nodes}[1]{\mathcal N(#1)}
\newcommand{\set}[2]{\left\{ #1\,\middle|\,#2\right\}}
\newcommand{\gtrace}{\bm\gamma_{\perp}}
\newcommand{\ptrace}{\bm\gamma_{\hspace{2pt}\parallel}}
\newcommand{\Rortho}{\bm R_{\perp}}
\newcommand{\Rpara}{\bm R_{\parallel}}
\newcommand{\Rtrace}{\bm R_{\bm\gamma}}
\newcommand{\tperp}{\bm t_{\perp}}

\newcommand{\ka}{T}

\newcommand{\de}{\,{\rm ds}}
\newcommand{\ds}{\,{\rm d}\sigma}
\newcommand{\dO}{\,{\rm d}\Omega}
\newcommand{\VP}{\mathbb V(\mesh)}

\newtheorem{theorem}{Theorem}[section]
\newtheorem{lemma}[theorem]{Lemma}
\newtheorem{corollary}[theorem]{Corollary}
\newtheorem{proposition}[theorem]{Proposition}
\newtheorem{definition}[theorem]{Definition}
\newtheorem{remark}[theorem]{Remark}

\newcommand{\bnew}[1]{{{#1}}}
\newcommand{\rnew}[1]{{{#1}}}

\newcommand{\eT}{\mathcal E(T)}

\newcommand{\step}[1]{\noindent\raisebox{1.5pt}[10pt][0pt]{\tiny\framebox{$#1$}}\xspace}

\author{R. H. Nochetto}
\address{R. H. Nochetto, Department of Mathematics, University of Maryland, College Park MD 20742 ({\tt rhn@math.umd.edu})}

\author{B. Stamm}
\address{B. Stamm, Department of Mathematics, University of California, Berkeley and Mathematics Department, Lawrence Berkeley National Laboratory, Berkeley 94720 ({\tt stamm@math.berkeley.edu})}
\title[A posteriori error estimates for EFIE on polyhedra]{A posteriori error estimates for the Electric Field Integral Equation on polyhedra}
\begin{document}
\maketitle
\begin{abstract}
We present a residual-based a posteriori error estimate for the
Electric Field Integral Equation (EFIE) on a bounded polyhedron $\Omega$
with boundary $\Gamma$. The EFIE is a variational equation formulated in 
${\bm H^{-1/2}_{{\rm div}}(\Gamma)}$.
We express the estimate in terms of $L^2$-computable quantities 
and derive global lower and upper bounds (up to oscillation terms).
\end{abstract}

\section{Introduction}
The Electric Field Integral Equation (EFIE) describes the scattering
of electromagnetic waves on a perfectly conducting obstacle $\Omega$
with surface $\Gamma$, in our case a polyhedron. Assuming a 
time-harmonic dependence, the Stratton-Chu
representation formula expresses the electric field $E$ 
in terms of a surface potential as
\[
E(\x) = E^{inc}(\x) + \int_\Gamma \left(
	\Gk(\x,\y)\u(\y)
	+ \frac{1}{k^2}
	\gradx\Gk(\x,\y) \div \u(\y)
	\right)\ds(\y),
\]
where $k$ denotes the wave-number and
$E^{inc}(\x) $ is the given incident wave that is scattered on $\Gamma$.
Invoking the boundary condition that the tangential component of the
total electric field $E$ vanishes on the surface $\Gamma$,
as corresponds to $\Omega$ being
perfectly conducting, the EFIE consists of seeking the surface current
$\u\in\Hdiv$ such that for all $\x\in\Gamma$
\[
	\int_\Gamma \left(
	\Gk(\x,\y)\u(\y)
	+ \frac{1}{k^2}
	\gradx\Gk(\x,\y) \div \u(\y)
	\right)\ds(\y)
	= -\ptrace(E^{inc})(\x),
\]  
where $\Hdiv$ is the space of traces of $\HcurlO$ functions that are
rotated by a right angle on the surface and $\ptrace$ denotes the
tangential trace onto $\Gamma$.

Computing approximations of the EFIE by means of the Boundary Element
Method (BEM), namely using a Galerkin approach based on the
variational formulation of the EFIE, is expensive due to the dense
matrix structure of the ensuing linear system. 
Although fast techniques such as the Fast Multipole Method exist,
c.f. \cite{Greengard:1987p5362} as an example of a first work in this
field, it is still crucial to 
locate the degrees of freedom efficiently, namely in regions of low
regularity of the solution $\u$.

Since $\u\in\Hdiv$, $\u$ exhibits in general rather low regularity and,
as a consequence, {\it a priori} estimates show extremely low
convergence rates for quasi-uniform mesh refinements;
see \cite{Christiansen:2004p270,Hiptmair:2002p479}.
In contrast,  
adaptive refinement techniques, based on {\it a posteriori} error
estimates, exploit much weaker regularity of $\u$ in a nonlinear 
Sobolev scale and allow for optimal error decay in terms of degrees
of freedom in situations where quasi-uniform meshes are suboptimal.
The design and analysis of a posteriori error estimators is, however,
problem dependent; we refer to \cite{Cascon:2008p964,NoSiVe:09} for an
account of the theory of adaptive finite element methods in the energy norm 
for linear second order elliptic \rnew{partial differential} equations in polyhedra.

\rnew{The a posteriori error analysis and corresponding theory of adaptive mesh
  refinements for BEM is much less developed
and an overview of different approaches for the former is given 
in \cite{Carstensen:2001p5368,Erath:2009p5528,Nochetto:2010p6914}.
It seems that this is the first contribution with specific application
to electromagnetic scattering problems on polyhedra.}

\bnew{For integral equations, additional difficulties arise since 
the residual typically lies in a Sobolev space with fractional index that is possibly also negative, as in the present case. 
Since such norms are not computable in practice, this imposes additional challenges to the residual based approach of a posteriori error estimates.}

In this paper we develop nevertheless a residual based a posteriori error 
estimator for the EFIE on polyhedra, and prove upper and lower global
bounds. Residual based estimators are especially attractive due to their
simplicity of derivation and computation, but they involve 
interpolation constants which can at best be estimated. Alternative
estimators have been proposed, mostly for elliptic problems defined in
$\Omega$, at the expense of their simplicity; we believe that our 
approach can be extended to those estimators as well.
We derive computable $L^2$--integrable quantities to estimate the
error of the BEM measured in the $\Hdiv$ norm, which is
the natural norm for EFIE.
We therefore avoid evaluating fractional Sobolev norms.

For proving well-posedness of the exact solution and developing {\it a
  priori } error estimates it is important to decompose both the exact
solution and test function using a Helmholtz
decomposition \cite{Buffa:2003p483,Hiptmair:2002p479}.
In contrast, to derive {\it a posteriori } error estimates, it
is crucial to decompose the {\it test } function according to
a regular decomposition which extends the Helmholtz
decomposition; see \cite{CaNoSi:07} for $H(\text{div};\Omega)$.

This paper is organized as follows. In Section \ref{sec:fsado} we
recall the necessary functional analysis in order to derive a
posteriori error estimates for the EFIE 
\cite{Buffa:2001p1125,Buffa:2001p1910,Buffa:2002p368,Buffa:2002p1911}.
We also present and study a Cl\'ement type interpolation
operator for the Raviart-Thomas space, based on ideas from 
\cite{Bernardi:2007p3092}.
We discuss the EFIE integral equation in Section \ref{sec:ps},
and derive global upper and lower a posteriori error estimates in Section
\ref{sec:apost}. Section
\ref{sec:conc} is finally left for conclusions.

\section{Functional Spaces and Differential Operators}
\label{sec:fsado}
The functional analysis framework developed in \cite{Buffa:2001p1125,Buffa:2001p1910} will be used in this work. In this section we give a short introduction to the functional spaces and differential operators used in the following sections. However, for a detailed and thorough overview we refer to \cite{Buffa:2001p1125,Buffa:2001p1910,Buffa:2002p1911,Buffa:2003p483}. References \cite{Buffa:2002p1911,Buffa:2003p483} deal with non-smooth Lipschitz surfaces, thus the theory is also valid for polyhedra, and covers therefore a more general framework. However, we restrict our theory to polyhedral surfaces.

\subsection{Spaces, norms and trace operators}
Let $\Omega$ be a bounded polyhedron in $\real^3$, and denote its
boundary by $\Gamma$ and its different faces by $\Gamma_j$,
$j=1,\ldots,N_F$. The exterior part $\Omega^+$ is defined by
$\Omega^+=\real^3\backslash \overline{\Omega}$. Let $\n(\x)$,
$\x\in\Gamma$, denote the outer unit normal to the surface $\Gamma$,
which is piecewise constant on $\Gamma$. We also indicate by
$e_{ij}=\partial\Gamma_i\cap\partial\Gamma_j$ the {\it edges} of
$\Gamma$ and by $\bm\tau_{ij}$ the unit vectors parallel to $e_{ij}$, 
with its orientation fixed but arbitrary.
If $\n_i = \n|_{\Gamma_i}$, we further define
\[
\bm\tau_i=\bm\tau_{ij}\times\n_i,
\quad
\bm\tau_j=\bm\tau_{ij}\times\n_j
\]
to be unit vectors lying on the supporting planes of $\Gamma_i$ and
$\Gamma_j$; see Figure \ref{fig:Gamma} for an illustration.

\begin{figure}[h!]
\centering
   \def\svgwidth{0.5\textwidth}
     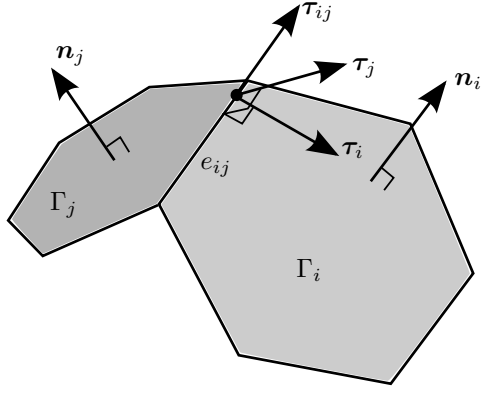
	\caption{Local coordinate systems around an edge
          $e_{ij}=\partial\Gamma_i\cap\partial\Gamma_j$.}
	\label{fig:Gamma}
\end{figure}

On $\Gamma$, we define the space of square integrable tangential fields
\[ 
	\Lt = \{\v\in [L^2(\Gamma)]^3 \,|\,\v\cdot \n=0\quad \rm{a.e.}\}.
\]
Moreover, we let $H^s(\Gamma)$ and $\Hb{s} = [H^s(\Gamma)]^3$, with
$s\in[-1,1]$, denote the standard Sobolev spaces of complex-valued
scalar and vector-valued functions on $\Gamma$ and denote their norms
by $\nH{s}{\cdot}$ and $\nHb{s}{\cdot}$, respectively;
for negative Sobolev indices the norms are defined by
duality. Furthermore, for $s\in(0,1)$, we denote by %
%
\[
\gamma:H^{s+\frac12}(\Omega)\rightarrow H^{s}(\Gamma),
\quad
\bm\gamma:[H^{s+\frac12}(\Omega)]^3\rightarrow \Hb{s}
\]
the standard continuous trace operators, and by
$R_\gamma$ and $\Rtrace$ their continuous right inverses.


For complex-valued vector functions we introduce the facewise $\bm
H^\frac12$-broken space
\[
	\bm H^\frac12_-(\Gamma) = \{\v\in\Lt \,|\,\v_{|\Gamma_i}\in \bm H^\frac12(\Gamma_i),\,1\le i \le N_F\}.
\]
with corresponding norm
\[
	\|\v\|^2_{\bm H_-^\frac12(\Gamma)}
	=\sum_{j=1}^{N_F} \|\v\|^2_{\bm H^\frac12(\Gamma_j)}.
\]
Moreover, we define the spaces
\begin{equation}\label{spaces}
\begin{aligned}
	\Hparas{\frac12} &= \set{\v\in \bm H^\frac12_-(\Gamma)}{\v_{|\Gamma_i}\cdot\bm\tau_{ij} \stackrel{1/2}{=} \v_{|\Gamma_j}\cdot\bm\tau_{ij},\text{ for every edge }e_{ij}},\\
	\Horthos{\frac12} &= \set{\v\in \bm H^\frac12_-(\Gamma)}{\v_{|\Gamma_i}\cdot\bm\tau_{i}\stackrel{1/2}{=}\v_{|\Gamma_j}\cdot\bm\tau_{j},\text{ for every edge }e_{ij}},
\end{aligned}
\end{equation}
where the relation $\stackrel{1/2}{=}$ is understood in the sense that
\begin{equation}
\label{eq:half_rel}
	v_i \stackrel{1/2}{=} v_j\quad\Leftrightarrow\quad 
	\int_{\Gamma_i}\int_{\Gamma_j}
	\frac{|v_i(\x)-v_j(\y)|^2}{\|\x-\y\|^3} \ds(\x) \ds(\y)<\infty.
\end{equation}
We further define
\begin{align*}
	\mathcal N^\parallel_{ij}(\v) &:= \int_{\Gamma_i}\int_{\Gamma_j}
	\frac{|(\v_{|\Gamma_i}\cdot\bm\tau_{ij})(\x)-(\v_{|\Gamma_j}\cdot\bm\tau_{ij})(\y)|^2}{\|\x-\y\|^3} \ds(\x) \ds(\y),\\
	\mathcal N^\perp_{ij}(\v) &:= \int_{\Gamma_i}\int_{\Gamma_j}
	\frac{|(\v_{|\Gamma_i}\cdot\bm\tau_{i})(\x)-(\v_{|\Gamma_j}\cdot\bm\tau_{j})(\y)|^2}{\|\x-\y\|^3} \ds(\x) \ds(\y),
\end{align*}
for each edge $e_{ij}$ of the polyhedron and denote by $\mathcal I_j$ the set of indices $i$ such that $\Gamma_j$ and $\Gamma_i$ have a common edge $e_{ij}$. 

\begin{proposition}[{\cite[Prop. 2.6]{Buffa:2001p1125}}]\label{P:H-1/2}
The spaces $\Hparas{\frac12}$ and $\Horthos{\frac12}$ are Hilbert spaces when endowed with the norms
\begin{align}
	\label{eq:paranorm}
	\nHpara{}{\v}^2 := \|\v\|^2_{\bm H_-^\frac12(\Gamma)} + \sum_{j=1}^{N_F}\sum_{i\in\mathcal I_j}\mathcal N^\parallel_{ij}(\v),\\
	\label{eq:orthonorm}
	\nHortho{}{\v}^2 := \|\v\|^2_{\bm H_-^\frac12(\Gamma)} + \sum_{j=1}^{N_F}\sum_{i\in\mathcal I_j}\mathcal N^\perp_{ij}(\v).
\end{align}
\end{proposition}

In other words, $\v\in\Hpara{},\Hortho{}$ satisfies $\v\in \bm
  H^\frac12(\Gamma_i)$ on the faces $\Gamma_i$ of $\Gamma$, and  
  the parallel resp. orthogonal component
  of the function $\v$ to edges $e_{ij}$ of $\Gamma$ are ``$H^\frac12$-continuous" in the sense of \eqref{eq:half_rel}; c.f. \cite{Buffa:2001p1125} for further details. 
 
We denote by $\Hpara{-}$, $\Hortho{-}$ the dual spaces of $\Hpara{}$, $\Hortho{}$ with pivot space $\Lt$. The corresponding duality pairing is denoted by $\para{\cdot}{\cdot}$ resp. $\ortho{\cdot}{\cdot}$.
The norms $\nHpara{-}{\cdot}$ and $\nHortho{-}{\cdot}$ are defined 
by duality.

For complex-valued functions $\v\in
[C^\infty(\overline{\Omega})]^3$ the tangential traces are
defined by
\begin{equation}\label{traces}
	\ptrace(\v):=\n\times(\v\times\n)_{|_\Gamma},
	\qquad	
	\gtrace(\v):=(\v\times\n)_{|_\Gamma}.
\end{equation}
We point out that $\ptrace(\v) = \v - (\v\cdot\n) \n$ gives the
component of $\v$ tangential to $\Gamma$, whereas $\gtrace(\v)$
provides a tangent vector field perpendicular to $\ptrace(\v)$.
Since $\Gamma$ is a polyhedron, for any edge $e_{ij}$ of $\Gamma$
the components of $\ptrace(\v)$ and $\gtrace(\v)$  tangential and normal to $e_{ij}$
are continuous, namely,
\begin{equation}\label{continuity}
\ptrace(\v)|_{\Gamma_i}\cdot\bm\tau_{ij}-\ptrace(\v)|_{\Gamma_j}\cdot\bm\tau_{ij}=0,
\qquad
\gtrace(\v)|_{\Gamma_i}\cdot\bm\tau_i-\gtrace(\v)|_{\Gamma_j}\cdot\bm\tau_j=0.
\end{equation}
This
means that both operators $\ptrace$ and $\gtrace$ can be viewed as
face-by-face projections; see \cite[p.36]{Buffa:2001p1125}.
Combining this observation with definitions \eqref{spaces},
we realize that $\Hpara{}$ and $\Hortho{}$ are the trace spaces 
of $\ptrace$, $\gtrace$ acting on $\Hone$. 
This is stated in the following Proposition.
\begin{proposition}[{\cite[Prop. 2.7]{Buffa:2001p1125}}]
\label{prop:otrace}
The trace operators
\[
\ptrace:\Hone\rightarrow\Hpara{},
\quad
\gtrace:\Hone\rightarrow\Hortho{}
\]
are linear, surjective and continuous operators.
In addition, there exists continuous right inverse maps 
$\Rpara:\Hpara{}\to \Hone$ and $\Rortho:\Hortho{}\to \Hone$.
\end{proposition}

We can now establish a critical result for the upcoming analysis.
Note that $\Hortho{}$ consists of tangential vector fields
  whereas $\Hb{\frac12}$ does not.

\begin{lemma}
\label{lem:ttrace}
There exists a continuous map $\tperp:\Hb{\frac12}\to \Hortho{}$ with right inverse $\tperp^{-1}:\Hortho{}\to \Hb{\frac12}$.
\end{lemma}
\begin{proof}
We define $\tperp:\Hb{\frac12}\to \Hortho{}$ and
$\tperp^{-1}:\Hortho{}\to \Hb{\frac12}$ by
\begin{equation*}
\begin{gathered}
\tperp(\w) = \gtrace(\Rtrace(\w)),
\quad\forall\,\w\in\Hb{\frac12},
\\
\tperp^{-1}(\v)=\bm\gamma(\Rortho(\v)),\quad\forall\,\v\in\Hortho{}
\end{gathered}
\end{equation*}
where $\gamma$ and $\Rtrace$ are the trace and its right inverse,
whereas $\gamma_\perp$ and $\Rortho$ are the operators of
Proposition \ref{prop:otrace}.
The continuity of these operators implies the continuity of $\tperp$
and $\tperp^{-1}$.

To prove that $\tperp^{-1}$ is the right inverse of $\tperp$,
we observe that
\[
\tperp(\tperp^{-1}(\v)) = \gtrace(\Rtrace(\bm \gamma(\Rortho(\v)))), \qquad \forall\v\in\Hortho{},
\]
and that $\gtrace$ projects face by face on $\Gamma$ \cite[page 36]{Buffa:2001p1125}.
If $\w=\tperp^{-1}(\v)\in \Hb{\frac12}$ and $\bm g=\Rtrace\w\in\Hone$,
then $\w=\gamma(\Rortho(\v))$ and $\tperp(\w)=\gtrace(\bm g)$.
Since
$\gtrace(\bm g)|_{\Gamma_i} =\bm\gamma(\bm g)_{|_{\Gamma_i}}\times\n$
for each face $\Gamma_i$ of $\Gamma$, we obtain on $\Gamma_i$
\begin{align*}
	\tperp(\tperp^{-1}(\v)) 
	& = \tperp(\w) = \gtrace(\bm g)
	= \bm\gamma(\bm g)\times\n \\
        & = \bm\gamma(\Rtrace\w)\times\n
        = \w\times\n
        = \bm\gamma(\Rortho(\v))\times\n
	= \gtrace(\Rortho(\v)) = \v.
\end{align*}
Thus, $\tperp^{-1}$ is indeed the right inverse of $\tperp$.
\end{proof}


\subsection{Tangential differential operators}
We set $H^{\frac32}(\Gamma) := \gamma(H^2(\Omega))$,
and define the tangential operators
$\grad:\H{\frac32}\rightarrow\Hpara{}$ and
$\curlv:\H{\frac32}\rightarrow\Hortho{}$ by
\begin{equation}
	\label{eq:gradcurl1}
	\grad \phi := \ptrace({\bf grad} \,\phi)\qquad
	\curlv \phi := \gtrace({\bf grad} \,\phi)\qquad\forall \phi\in H^2(\Omega),
\end{equation}
where ${\bf grad}$ denotes the standard gradient in $\real^3$. 
According to definitions \eqref{traces}, $\grad \phi$ is the
orthogonal projection of ${\bf grad} \phi$ on each face $\Gamma_i$ of
$\Gamma$, whereas $\curlv \phi$ is obtained from the former by a 
$\pi/2$ rotation.
It can be shown that the maps $\grad$ and $\curlv$ are linear and continuous. 

The adjoint operators $\div:\Hpara{-}\rightarrow \H{-\frac32}$ and
$\curls:\Hortho{-}\rightarrow \H{-\frac32}$ can be defined as follows
\begin{align*}
	\threetwo{\div \v}{\phi} &= -\para{\v}{\grad \phi},\\
	\threetwo{\curls \w}{\phi} &= \ortho{\w}{\curlv \phi},
\end{align*}
for all $\phi\in \H{\frac32}$, $\v\in\Hpara{-}$ and $\w\in\Hortho{-}$.

In view of these definitions we now introduce the spaces
\begin{align*}
	\Hdiv &:= \set{\v\in\Hpara{-}}{\div \v\in \H{-\frac12}},\\
	\Hcurl &:= \set{\v\in\Hortho{-}}{\curls \v\in \H{-\frac12}},
\end{align*}
which will play a crucial role for the upcoming analysis and are 
endowed with the graph norms
\begin{align*}
	\nHdiv{\v}^2 &:= \nHpara{-}{\v}^2 + \nH{-\frac12}{\div \v}^2,\\
	\nHcurl{\v}^2 &:= \nHortho{-}{\v}^2 + \nH{-\frac12}{\curls \v}^2.
\end{align*}
Let the natural solution space of Maxwell's equations be denoted by 
\[
	\HcurlO=\set{\v\in \bm L^2(\Omega)}{\curl \v\in\bm L^2(\Omega)}.
\]

\begin{theorem}[{\cite[Theorem 4.6]{Buffa:2001p1910}}]
The mappings $\ptrace$ and $\gtrace$ admit linear and continuous
extensions
\[
\ptrace:\HcurlO\rightarrow \Hcurl,
\quad
\gtrace:\HcurlO\rightarrow \Hdiv.
\]
Moreover, the following integration by parts formula holds true:
\begin{equation*}
	\int_\Omega \Big(\curl \v\cdot \u -\curl \u\cdot\v\Big)\dO = \para{\gtrace\u}{\ptrace\v},\quad \forall \,\u\in\HcurlO,\v\in\Hone.
\end{equation*}
\end{theorem}

Furthermore, a duality pairing $\po{\cdot}{\cdot}$ between $\Hdiv$ and
$\Hcurl$ can be established by using an orthogonal decomposition of
those spaces so that the following integration by parts formula still holds
\begin{equation*}
	\label{eq:IntByParts}
	\int_\Omega \Big(\curl \v\cdot \u -\curl \u\cdot\v\Big)\dO = \po{\gtrace\u}{\ptrace\v},\qquad \forall \,\u,\v\in\HcurlO.
\end{equation*}
For more details, we refer to \cite{Buffa:2001p1910}.

The differential operators $\grad$ and $\curlv$ can be further
extended as follows.
\begin{proposition}[{\cite[p.39]{Buffa:2001p1910}}]\label{P:grad-curl}
The tangential gradient and curl operators introduced in
\eqref{eq:gradcurl1} can be extended to linear and continuous
operators defined on $\H{\frac12}$
	\[
		\grad:\H{\frac12}\rightarrow \Hortho{-},\qquad
		\curlv:\H{\frac12}\rightarrow \Hpara{-}.
	\]
\end{proposition}
Moreover their formal $\Lt$-adjoints 
\[
	\div:\Hortho{}\rightarrow\H{-\frac12}
	\qquad{\rm and}\qquad
	\curls:\Hpara{}\rightarrow \H{-\frac12}
\]
can be defined by
\begin{equation}\label{div-curl}
\begin{aligned}
	\half{\div \v}{\phi} &= -\ortho{\v}{\grad \phi},\\
	\half{\curls \w}{\phi} &= \para{\w}{\curlv \phi},
\end{aligned}
\end{equation}
for all $\phi\in \H{\frac12}$, $\v\in\Hortho{}$ and $\w\in\Hpara{}$.

\subsection{Potentials}\label{S:potentials}
Let $\Gk$ denote the fundamental solution of the Helmholtz operator
$\Delta+k^2$, which is given by
\[
	\Gk(\x,\y) := \frac{\exp(ik|\x-\y|)}{4\pi|\x-\y|}.
\]
The scalar and vector single layer potential are then defined respectively by
\begin{align*}
	&\SLPs:\H{-\frac12}\rightarrow H_{{\rm loc}}^1(\real^3),  &\SLPs(v)(\x) &:= \int_\Gamma \Gk(\x,\y)v(\y)\ds(\y),\\
	&\SLPv:\Hpara{-}\rightarrow \bm H_{{\rm loc}}^1(\real^3), &\SLPv(\v)(\x) &:= \int_\Gamma \Gk(\x,\y)\v(\y)\ds(\y).
\end{align*}
These potentials are known to be continuous \cite[Theorem 3.8]{Buffa:2003p483}.
Finally the scalar and vector single layer boundary operators are defined by
\begin{align*}
	&\Vk:\H{-\frac12}\rightarrow \H{\frac12},  &\Vk &:= \gamma\circ\SLPs,\\
	&\Ak:\Hpara{-}\rightarrow \Hpara{}, &\Ak &:= \ptrace\circ\SLPv.
\end{align*}
The simultaneous continuity of the trace operators $\gamma$, $\ptrace$
and the single layer potentials yield then the continuity of the
single layer boundary operators, namely,
\begin{equation}
	\label{eq:continuitySLBP}
	\nH{\frac12}{\Vk v} \preceq \nH{-\frac12}{v},
	\qquad
	\nHpara{}{\Ak \v} \preceq \nHpara{-}{\v},
\end{equation}
for all $v\in\H{-\frac12}$ and $\v\in\Hpara{-}$.
In particular, if restricted to
  $L^2(\Gamma)\subset\H{-\frac12}$, then the range of $\Vk$ lies in $\H{1}$ 
\cite[Theorem 3.8]{Buffa:2003p483}, i.e.,
\begin{equation}
	\label{eq:RL2Vk}
	{\rm Im}(\Vk(L^2(\Gamma))\subset \H{1}.
\end{equation}
Likewise, for the vector case the corresponding result reads
\cite[Prop. 2]{Buffa:2002p368}
\begin{equation}
	\label{eq:RL2Ak}
	{\rm Im}(\Ak(\Lt)\subset \bm H^1(\Gamma).
\end{equation}

\subsection{Interpolation of weighted spaces}
In the following section we will be confronted with interpolation of
weighted $L^2$-spaces. We thus recall in this section some basic
results taken from Tartar's book \cite{MR2328004}, which are valid
without regularity on the weights.

Let $\mesh$ be a family of shape-regular triangulations decomposing $\Gamma$ into flat triangles such that the surface covered by the triangles coincides with $\Gamma$. Denote the set of edges of the mesh by $\edges$. For a fixed triangulation let $h_\ka$ denote the diameter of any element $\ka\in\mesh$ and let $\h$ be the piecewise constant function such that $\h|_\ka=h_\ka$.

\begin{lemma}[{\cite[Lemma 22.3, p.110]{MR2328004}}]
	\label{lem:Tartar1}
	If $A$ is linear from $E_0+E_1$ into $F_0+F_1$ and maps $E_0$ into $F_0$ with $\|Ax\|_{F_0}\le M_0 \|x\|_{E_0}$ for all $x\in E_0$, and maps $E_1$ into $F_1$ with $\|Ax\|_{F_1}\le M_1 \|x\|_{E_1}$ for all $x\in E_1$, then $A$ is linear continuous from $(E_0,E_1)_{\theta,p}$ into $(F_0,F_1)_{\theta,p}$ for all $\theta,p$, and for $0<\theta<1$ one has
\[
	\|Aa\|_{(F_0,F_1)_{\theta,p}} \le M_0^{1-\theta} M_1^\theta \|a\|_{(E_0,E_1)_{\theta,p}}
	\quad\text{for all }a\in (E_0,E_1)_{\theta,p}.
\]
The space $(E_0,E_1)_{\theta,p}$ denotes the interpolation space between $E_0$ and $E_1$ based on the $L^p$--inner product.
\end{lemma}
\begin{lemma}[{\cite[Lemma 23.1, p.115]{MR2328004}}]
	\label{lem:Tartar2}
	For a measurable positive function $w$ on $\Gamma$, let
	\[
		E(w) = \left\{ u\,\middle|\,\int_\Gamma |u(x)|^2 w(x)\,dx<\infty\right\} 
		\quad \text{with} \quad
		\|u\|_w=\left(\int_\Gamma |u(x)|^2 w(x)\,dx \right)^{\frac12}.
	\]
	If $w_0,w_1$ are two such functions, then for $0<\theta<1$ one has
	\[
		(E(w_0),E(w_1))_{\theta,2} = E(w_\theta)
		\quad\text{with equivalent norms, where } w_\theta=w_0^{1-\theta}w_1^\theta.
	\]
\end{lemma}
\begin{corollary}
\label{cor:Interpolation1}
Let $0<s<1$ be arbitrary. 
Let $A$ be a linear continuous map from $L^2(\Gamma)$ into
$L^2(\Gamma)$ and from $H^1(\Gamma)$ into $L^2(\Gamma)$ with
\begin{equation*}
\begin{aligned}
\|Av\|_{L^2(\Gamma)} &\le M_0 \|v\|_{L^2(\Gamma)}\quad\text{for all } v\in
L^2(\Gamma),
\\
\|\h^{-1}Av\|_{L^2(\Gamma)} &\le M_1 \|v\|_{H^1(\Gamma)}\quad\text{for all } v\in
H^1(\Gamma).
\end{aligned}
\end{equation*}
Then $A$ is a linear map from
$H^{s}(\Gamma)=(H^1(\Gamma),L^2(\Gamma))_{s,2}$ into
$L^2(\Gamma)$ with
\[
	\|\h^{-s}Av\|_{L^2(\Gamma)} \le M_0^{1-s} M_1^s \|v\|_{H^s(\Gamma)}
	\quad\text{for all }v\in H^\frac12(\Gamma).
\]
\end{corollary}
\begin{proof}
	Combine Lemma \ref{lem:Tartar1} with Lemma \ref{lem:Tartar2}.
\end{proof}

\subsection{Discrete spaces and interpolation}
Let $\RT(\ka)$ denote the {\it local} Raviart-Thomas space of complex-valued functions on $\ka\in\mesh$ defined by (cf. \cite{Ern:2004p120,Raviart:1977p2571})
\[
	\RT(\ka):=\set{\bm v(\bm x) =\bm \alpha +\beta \bm x }{\bm\alpha\in \complex^2, \beta\in\complex}.
\]
The {\it global} Raviart-Thomas space is defined by
\[
	\RT := \set{\bm v\in \Hdivp0}{ {\bm v}_{|\ka} \in \RT(\ka)\quad \forall \ka\in \mesh},
\]
where $\Hdivp0$ is defined in a standard manner
\[
	\Hdivp0:=\set{\bm v\in \Lt}{\div \bm v\in L^2(\Gamma)}.
\]
Further denote by $V_\mesh$ the space of scalar complex-valued
continuous functions that are piecewise linear, namely
\begin{equation}\label{pw-linears}
	\VP = \set{v\in \H{1}}{v|_K\in \mathbb P_1(\ka)},
\end{equation}
where $\mathbb P_1(\ka)$ denotes the space of affine polynomials on $\ka$.
Let $\nodes{\mesh}$ denote the set of all nodes $\nu$ of $\mesh$
and $\{\varphi_\nu\}_{\nu\in\nodes{\mesh}}$ be the family of nodal bases of $\VP$.
\begin{definition}
	\label{def:interpolationVh}
Let the Cl\'ement type interpolation operator $\Ialpha:L^2(\Gamma)\rightarrow \VP$ be\looseness=-1
	\[
		\Ialpha v := \sum_{\nu\in\nodes{\mesh}} \phi_\nu(v)\varphi_\nu\qquad\forall v\in L^2(\Gamma)
	\]
	where $\Gamma_\nu = {\rm supp} (\varphi_\nu)$ and the degrees of freedom are given by
	\[
		\phi_\nu(v):= \frac{3}{|\Gamma_\nu|}\int_{\Gamma_\nu} v(\x) \varphi_\nu(\x)\,d\x.
	\]
\end{definition}
\begin{proposition} If $v\in\H{s}$ with $0<s<1$, 
then the interpolation operator $\Ialpha$ satisfies the following
interpolation properties 
	\begin{align}
		\label{eq:Interp2}
		\nlt{\h^{-s}(v-\Ialpha v)}{\Gamma} \preceq \nH{s}{ v}
                \quad\text{for all } v\in\H{s}.
	\end{align}
\end{proposition}

\begin{proof}
	This interpolation operator $\Ialpha$ is also used in \cite{Demlow:2007p1912} and the following result is proven
	\[
		\nlt{\h^{-1}(v-\phi_\nu(v))}{\Gamma_\nu} \preceq  \nLt{\grad v }{\Gamma_\nu}
	\]	
	for $v\in H^1(\Gamma)$ under the assumption of
        shape-regularity \cite[(2.2.29)]{Demlow:2007p1912} (Note that
        in our case the mesh matches the surface and therefore
        equation (2.2.29) can be simplified). 
        Following the arguments of the original paper of Cl\'ement
        \cite[Proof of Theorem 1]{Clement:1975p227},
	it is now straightforward to prove that
	\[
		\nlt{\h^{-1}(v-\Ialpha v)}{\ka} \preceq \sum_{\nu\in \nodes{T}} \nLt{\grad v }{\Gamma_\nu},
	\]
	where $\nodes{T}$ denotes the set of nodes of the element
        $\ka$. Now, summing over all elements of the mesh $\mesh$ and
        using that the number of elements sharing a node is
          bounded, as a consequence of shape regularity of $\mesh$, we get
	\[
		\nlt{\h^{-1}(v-\Ialpha v)}{\Gamma} \preceq  | v |_{H^1(\Gamma)}.
	\]
	Furthermore, the operator can also be shown to be $L^2$-stable \cite[(2.2.33)]{Demlow:2007p1912}.
	
	Therefore, the linear continuous operator $A_{\mesh}={\rm Id}-\Ialpha :L^2(\Gamma)\mapsto L^2(\Gamma)$ satisfies
	\begin{align*}
		\nlt{\h^{-1}(v-\Ialpha v)}{\Gamma} &\preceq  \| v \|_{H^1(\Gamma)}, \\
		\nlt{v-\Ialpha v}{\Gamma} &\preceq  \nlt{ v }{\Gamma}.
	\end{align*}
	The asserted estimate \eqref{eq:Interp2} follows from Corollary \ref{cor:Interpolation1}.
\end{proof}

Besides this for $s=1/2$,
we will also need a Raviart-Thomas type interpolation operator for functions in $\v\in\Hortho{}$. Since $\div \v\notin L^2(\Gamma)$, the standard degrees of freedom
are no longer well-defined. Therefore, we will utilize an interpolation operator similar to that introduced in \cite{Bernardi:2007p3092} for the first type N\'ed\'elec elements. 

For any edge $e\in\edges$ of the mesh we associate an arbitrary but
fixed element $\ka_e$ such that $e\subset \partial \ka_e$. On $\ka_e$
we denote by $\bm\pi_e$ the $\bm L^2(\ka_e)$-projection onto constant
functions. We let $\{\bm\psi_e\}_{e\in\edges}$ be the standard
Raviart-Thomas basis of lowest order, sometimes also referred to as the Rao-Wilton-Glisson (RWG) basis in this context of electromagnetic scattering, such that
\[
	\int_e \bm\psi_e\cdot \bm\nu_e \de=1\qquad\text{and}\qquad\int_e \bm\psi_{e'}\cdot \bm\nu_e \de=0
\]
for any $e'\in\edges$ such that $e\neq e'$ and where $\bm\nu_e$
denotes the outer unit normal of $\ka_e$ at the edge $e$ which is
coplanar with $T_e$; see Fig \ref{fig:dof}.

\begin{definition}
	\label{def:interpolationRT}
Let the Cl\'ement type interpolation operator
  $\Ipsi:\Lt\rightarrow \RT$ for the Raviart-Thomas element of lowest
  order be given by 
	\[
	\Ipsi \v:=\sum_{e\in\edges}\alpha_e(\v) \bm\psi_e
	\]
	where the degrees of freedom are defined by
	\[
		\alpha_e(\v):=\int_e \bm\pi_e(\v)\cdot \bm\nu_e \de.
	\]
\end{definition}
\begin{remark}
	Note that $\Ipsi$ does not satisfy the usual commutative property
	\[
		\div (\Ipsi\v)\neq  {\rm P}_0 (\div \v),
	\]
	where $ {\rm P}_0$ denotes the element-wise $L^2$-projection
        of degree 0.
        This is important in the a priori analysis
	but not in the upcoming a posteriori error analysis.
\end{remark}

\begin{figure}
\centering
   \def\svgwidth{0.6\textwidth}
     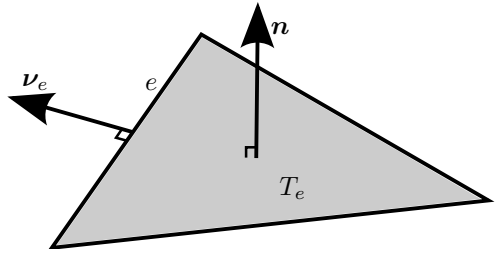
     \label{fig:dof}
     \caption{Illustration of the normals on an element $\ka_e$.}
\end{figure}


\begin{lemma}
	\label{lem:WPRTIO}
	The degrees of freedom of the interpolation operator
        $\Ipsi:\Lt\rightarrow \RT$ are well-defined and $\Ipsi$
          satisfies the local $L^2$-stability bound
	\[
		\nLt{\Ipsi\v}{\ka}
		\le c \,\nLt{\v}{\Delta_\ka}
                \quad\text{for all } T\in\mesh,
	\]
	where $\Delta_\ka$ denotes the set of elements that share at least one edge with $\ka$.
\end{lemma}

\begin{proof}
The argument is similar to \cite{Bernardi:2007p3092}. If 
$\ka\in\mesh$ is an arbitrary but fixed element and $\eT$ 
denotes the three edges of $T$, then
\begin{equation*}
	\nLt{\Ipsi\v}{\ka} = \nLt{\sum_{e\in\eT}\alpha_e\bm\psi_{e}}{\ka}
	\le \sum_{e\in\eT}\nLt{\alpha_e\bm\psi_{e}}{\ka}
	= \sum_{e\in\eT}|\alpha_e|\,\nLt{\bm\psi_{e}}{\ka}.
\end{equation*}
Invoking the Piola transformation, it can be shown that
\begin{equation*}
	\nLt{\bm\psi_{e}}{\ka} 
	\le c\,\nLt{\widehat{\bm\psi}_{\hat e}}{\hat \ka} 
	\le c
\end{equation*}
since the basis functions $\widehat{\bm\psi}_{\hat e}$ on the
reference element $\widehat \ka$ are uniformly bounded independent on
$\ka$. Moreover, applying the Cauchy-Schwarz inequality, we get

\[
	|\alpha_e|=\left|\int_{e} \bm\pi_{e}(\v)\cdot \bm\nu_{e} \de\right|\,
	\le c h_\ka^\frac12 \,\nLt{\bm\pi_{e}\v}{e}
	\le c h_\ka \nLt{\hat{\bm\pi}_{e}\hat\v}{\hat e}
\]
where $\hat{\bm\pi}_e$ denotes the $\bm L^2(\widehat \ka)$-projection
onto constant functions on the reference element $\widehat T$. 
Note
that $\hat{\bm\pi}_e\hat\v$ is defined via the affine transformation
from $\ka_e$ (and not $\ka$) to $\hat\ka$.
Norm equivalence of polynomials (constant functions in this case),
the $L^2$-stability of $\hat{\bm\pi}_e$ and a scaling argument yield
\begin{equation}
	\label{eq:bb2}
	|\alpha_e|\,
	\le c h_\ka \nLt{\hat{\bm\pi}_e\hat\v}{\widehat \ka}
	\le c h_\ka \nLt{\hat\v}{\widehat \ka}
	\le c \,\nLt{\v}{\ka_e}.
\end{equation}
Combining the above estimates implies the asserted stability bound of $\Ipsi$.
\end{proof}

To explore the accuracy of the interpolant $\Ipsi$, we need the following lemmas.

\begin{lemma}[Local approximability]
\label{lem:locapp}
For any $\bm c\in\real^3$, we have the error estimate
	\[
		\nLt{v-\Ipsi\v}{\ka}
		\preceq\nLt{\v-\gtrace\bm c}{\Delta_\ka}, 
                \quad\text{for all } \ka\in\mesh,
	\]
        where $\Delta_\ka$ denotes the set of elements that share at least one edge with $\ka$.
\end{lemma}
\begin{proof}
	Let $\bar{\v}=\gtrace\bm c$ with $\bm c\in\real^3$
which, in view of \eqref{traces}, is piecewise constant in
  $\mesh$.
According to \eqref{continuity} the normal component of $\bar{\v}$ is
continuous across all edges of the mesh, including those of the
polyhedron $\Gamma$, whence $\bar{\v}\in\RT$.

	We first observe that $\Ipsi \bar{\v}=\bar{\v}$ because 
	$\bm\pi_e\overline\v|_{\ka_e}=\overline\v|_{\ka_e}$ for any
        edge $e\subset\partial \ka$ and
	\[
		\alpha_e(\overline\v)=\int_e \bm\pi_e\overline\v \cdot\bm\nu_e \de
		= \int_e \overline\v\cdot\bm\nu_e\de.
	\]
Since these three local degrees of freedom on $T\in\mesh$ are
unisolvent and they coincide for 
both $\Ipsi\bar\v|_\ka$ and $\bar\v|_\ka$, we deduce 
$\Ipsi\overline\v|_\ka=\overline\v|_\ka$. Consequently
	\[
		\nLt{\v-\Ipsi\v}{\ka}\le \nLt{\v-\overline\v}{\ka} + \nLt{\Ipsi(\v-\overline\v)}{\ka}.
	\]
	By the local $\bm L^2$--stability of Lemma \ref{lem:WPRTIO} we conclude that
	\[
		\nLt{\v-\Ipsi\v}{\ka}
		\preceq\nLt{\v-\overline\v}{\Delta_\ka},
	\]
as asserted.
\end{proof}

\begin{lemma}
	\label{lem:vw}
	If $\v\in\Hortho{}$, then there exists $\w\in\Hb{\frac12}$ such that
	$\tperp(\w)=\v$ and $\nHb{\frac12}{\w}\preceq\nHortho{}{\v}  $.
\end{lemma}
\begin{proof}
Simply set $\w=\tperp^{-1}(\v)$, where $\tperp^{-1}$ is defined
  in Lemma \ref{lem:ttrace}, and use the facts that $\tperp^{-1}$ is the right inverse
  of $\tperp$ and the
  continuity of $\tperp^{-1}$.
\end{proof}

\begin{lemma}
	\label{lem:wL2}
	If $\w\in\Hb{\frac12}$, then $\nLt{\tperp(\w)}{\Delta_\ka}
\preceq \nLt{\w}{\Delta_\ka}$.
\end{lemma}
\begin{proof}
	As in the proof of Lemma \ref{lem:ttrace}, let $\bm g\in\Hone$ be the function such that $\bm g=\Rtrace\w$ and $\tperp(\w)=\gtrace(\bm g)$. Therefore 
\[
	\gtrace(\bm g)|_{T} = \bm\gamma(\bm g)|_{T}\times\n ,
\quad \text{for a.e. } \x\in T,\text{ for all } T\in \mesh,
\]
whence
\begin{align*}	
	\nLt{\tperp(\w)}{\Delta_\ka}^2
	&= \sum_{\ka\subset\Delta_\ka}\nLt{\tperp(\w)}{\ka}^2
	= \sum_{\ka\subset\Delta_\ka}\nLt{\bm\gamma(\bm g)\times\n}{\ka}^2\\
	&\preceq \sum_{\ka\subset\Delta_\ka}\nLt{\bm\gamma(\bm g)}{\ka}^2
	= \nLt{\bm\gamma(\bm g)}{\Delta_\ka}^2
	= \nLt{\w}{\Delta_\ka}^2
\end{align*}
because $\Rtrace$ is the right inverse of $\bm\gamma$ and thus
$\gamma(\bm g)=\w$.
\end{proof}

\begin{proposition}[Global approximability]
If $\v\in\Hortho{}$, then the interpolation operator $\Ipsi$
satisfies the following global error estimate
	\begin{align}
		\label{eq:Interp1}
		\nLt{\h^{-\frac12}(\v - \Ipsi \v)}{\Gamma} \preceq
                \nHortho{}{ \v} \quad\text{for all }\v\in\Hortho{}.
	\end{align}
\end{proposition}
\begin{proof}
Given $\v\in\Hortho{}$, there exists $\w\in\Hb{\frac12}$ so that
$\tperp(\w)=\v$ according to Lemma \ref{lem:vw}.
For each $\ka \in\mesh$, we define
$\overline{\w}_\ka=\int_{\Delta_\ka} \w(\x) d\x\in\real^3$ and
$\overline\v_\ka =\tperp(\overline\w_\ka)\in\Hortho{}$.
Since the estimate of Lemma \ref{lem:locapp} is local, we have
\[
	\nLt{\h^{\alpha}(\v-\Ipsi\v)}{\Gamma}^2
	= \sum_{\ka\in\mesh}\nLt{\h^{\alpha}(\v-\Ipsi\v)}{\ka}^2
	\preceq \sum_{\ka\in\mesh}\nLt{h_\ka^{\alpha}(\v-\overline\v_\ka)}{\Delta_\ka}^2
\]
for $\alpha=-1,0$. Using Lemma \ref{lem:wL2} yields
\[
\nLt{\v-\overline\v_\ka}{\Delta_\ka}
=\nLt{\tperp(\w-\overline\w_\ka)}{\Delta_\ka}
\preceq \nLt{\w-\overline\w_\ka}{\Delta_\ka}
\]
and stability of the $L^2$-projection together with the
definition of $\overline\w_\ka$ implies
\begin{align*}
	\nLt{\w-\overline\w_\ka}{\Delta_\ka} &\preceq \nLt{\w}{\Delta_\ka}, \\
	\nLt{h_\ka^{-1}(\w-\overline\w_\ka)}{\Delta_\ka} &\preceq \|\w\|_{\bm H^1(\Delta_\ka)},
\end{align*}
whence
\begin{align*}	
	\nLt{\v-\Ipsi\v}{\Gamma}^2	
	\preceq \sum_{\ka\in\mesh}\nLt{\w}{\Delta_\ka}^2
	\preceq \nLt{\w}{\Gamma}^2, \\
	\nLt{\h^{-1}(\v-\Ipsi\v)}{\Gamma}^2
	\preceq \sum_{\ka\in\mesh}\|\w\|_{\bm H^1(\Delta_\ka)}^2
	\preceq \|\w\|_{\bm H^1(\Gamma)}^2.
\end{align*}
Applying Corollary \ref{cor:Interpolation1} to vector-valued
functions, we obtain
\[
	\nLt{\h^{-\frac12}(\v-\Ipsi\v)}{\Gamma}
	\preceq \|\w\|_{\bm H^{\frac12}(\Gamma)}
	\preceq \nHortho{}{\v},
\]
where the last inequality results from Lemma \ref{lem:vw}. This
concludes the proof.
\end{proof}

\section{Problem Setting}
\label{sec:ps}
%
The variational formulation of the Electric Field Integral Equation
(EFIE), also called Rumsey principle, consists of seeking $\u\in\Hdiv$
such that 
\begin{equation}
	\label{eq:EFIE}
	a(\u, \v) = \po{\f}{\v}\qquad\text{for all }\v\in\Hdiv
\end{equation} 
where $\f \in \Hcurl$, the sesquilinear form $a(\cdot,\cdot)$ is given by
\[
a(\u, \v) := \half{\Vk \div \u}{\div \v} - k^2 \para{\Ak \u}{\v},
\]
$\po{\cdot}{\cdot}$ is the duality pairing between
$\Hcurl$ and $\Hdiv$, $\half{\cdot}{\cdot}$ is the duality
  pairing $H^{\frac12}(\Gamma)-H^{-\frac12}(\Gamma)$, 
$\para{\cdot}{\cdot}$ is the duality pairing 
$\Hparas{\frac12}-\Hparas{-\frac12}$,
and the integral operators
$\Vk,\Ak$ has been defined in \S \ref{S:potentials}.

\par
The discrete formulation reads: find $\uh\in\RT$ such that 
\begin{equation}
	\label{eq:Pb}
	a(\uh, \vh) = \po{\f}{\vh}\qquad\text{for all }\vh\in\RT.
\end{equation} 

Equation \eqref{eq:EFIE} is well-posed
under the assumption that the wave number $k$ does not correspond to
an interior eigenmode of the Maxwell problem on $\Gamma$. 
As a consequence, the following continuous inf-sup
condition holds (see also \cite{Hiptmair:2002p479}):
\begin{equation}
	\label{eq:InfSup}
	\nHdiv{\u} \preceq
        \sup_{\v\in\Hdiv}\frac{a(\u,\v)}{\nHdiv{\v}}
        \quad\text{for all }\u\in\Hdiv.
\end{equation}
Since the boundary element discretization is conforming,
i.e. $\RT\subset \Hdiv$, the following {\it Galerkin orthogonality} holds:
if $\u\in\Hdiv$ is the solution of \eqref{eq:EFIE} and $\uh\in\RT$ is 
the solution of \eqref{eq:Pb}, then
\begin{equation}\label{galerkin-ortho}
	a(\u-\uh,\vh)=0 \quad\text{for all }\vh\in\RT.
\end{equation}

In addition, as a direct consequence of the Cauchy-Schwarz inequality 
and the continuity \eqref{eq:continuitySLBP} of the single layer boundary operators,
the form $a(\cdot,\cdot)$ is continuous:
\begin{equation}\label{a-continuous}
	a(\v,\w)\preceq \nHdiv{\v}\nHdiv{\w}
        \quad\text{for all }\v,\w\in\Hdiv{}.
\end{equation}

\section{A Posteriori  Error  Analysis}
\label{sec:apost}
As is customary in the theory of a posteriori error estimation, one
has to assume a higher regularity of the right-hand side than it is
needed for well-posedness in order to derive computable error
bounds. Therefore we assume in this section that
\begin{equation}\label{assump-f}
\f\in\Hpara{}\cap
\bm H^0_{{\rm curl}}(\Gamma)
\end{equation}
with $\Hpara{}$ given in Proposition \ref{P:H-1/2} and
$H^0_{{\rm curl}}(\Gamma)=\set{v\in L^2(\Gamma)}{\curlv v\in\Lt}$.\looseness=-1

We proceed as in Casc\'on, Nochetto, and Siebert \cite{CaNoSi:07} for
flat domains. To this end, we start with 
some auxiliary results that will be useful for our analysis later.

\begin{lemma}[Regular decomposition {\cite[Theorem 5.5]{Buffa:2002p1911}}]
	\label{lem:HodgeDecomp}
	The decomposition
	\[
		\Hdiv = \curlv (H^\frac12(\Gamma)/\mathbb C)+\Hortho{}
	\]
	is valid and is stable, i.e. for $\v= \Psiv + \curlv \alpha$ with $\Psiv\in\Hortho{}$ and $\alpha\in H^\frac12(\Gamma)\backslash \complex$,
	\begin{equation}
		\label{eq:StabHD}
		\nHortho{}{\Psiv} + \nH{\frac12}{\alpha} \preceq
                \nHdiv{\v}
                \quad\text{for all }\v\in \Hdiv.
	\end{equation}
\end{lemma}

\begin{lemma}
	For $\VP$ given by \eqref{pw-linears} there holds
	$
		\curlv(\VP)\subset \RT.
	$
\end{lemma}
\begin{proof}
By  \cite[Corollary 5.3]{Buffa:2002p1911} we have
\[
	\ker(\div)\cap \Lt = \curlv(\H1).
\]
Thus for all $\alpha\in \VP\subset\H1$ we infer that 
$\curlv \alpha\in \Lt$ is piecewise constant and that 
$\div\,\curlv\alpha\equiv 0\in \lt$. This implies that $\curlv
\alpha\in \Hdivp{0}$.
\end{proof}

\subsection{Upper Bound}
%
Let $\u\in\Hdiv$ be the exact solution of \eqref{eq:EFIE} and $\uh\in\RT$ be its approximation defined by \eqref{eq:Pb}. By the Galerkin orthogonality observe that 
\[
	a(\u-\uh,\v)=a(\u-\uh,\v-\vh) \quad\text{for all }
        \v\in\Hdiv, \vh\in\RT.
\]
Decompose $\v$ as $\v= \Psiv + \curlv \alpha$, according to Lemma
\ref{lem:HodgeDecomp}, and define 
\[
	\dPsiv:=\Psiv-\Psivm,\qquad
	\dalpha:=\alpha-\alpham
\]
where $\Psivm\in\RT$ and $\alpham \in \VP$ can be arbitrarily chosen.
Thus we can write $\v-\vh= \dPsiv + \curlv \dalpha$ and
\begin{align*}
	a(\u-\uh &,\v-\vh)
	=\po{\f}{\dPsiv + \curlv \dalpha}-a(\uh,\dPsiv + \curlv \dalpha) \\
	&= \underbrace{\po{\f}{\dPsiv} + \para{k^2\Ak \uh}{\dPsiv}}_{=\calI_1} 
	\\ & +  \underbrace{\po{\f}{ \curlv \dalpha} 
        +  \para{k^2\Ak \uh}{ \curlv\dalpha} }_{=\calI_2} 
        - \underbrace{\half{\Vk \div \uh}{\div \dPsiv}}_{=\calI_3}
\end{align*}
for any $\v\in\Hdiv$, $\Psivm\in\RT$ and $\alpham\in \VP$.
We proceed in four steps.

\noindent
\\
\step{1} We note that $\f\in \Lt$, $k^2\Ak \uh\in
\Hpara{}\subset\Lt$ and that $\Psiv\in\Hortho{}\subset\Lt$ due to
enhanced regularity of $\Psiv$ asserted in 
Lemma \ref{lem:HodgeDecomp}. Since $\Psivm\in\RT\subset\Lt$, we can
replace the duality pairing in $\calI_1$ by an integral and thus
write
\begin{equation}
	\label{eq:I1}
	\calI_1 = \int_\Gamma (\f+k^2\Ak \uh)\cdot \dPsiv\ds.
\end{equation}

\step{2}  Since $\f\in\Hpara{}$ the duality pairing
$\po{\cdot}{\cdot}$ can be interpreted as
\[
	\po{\f}{ \curlv \dalpha} = 
	\para{\f}{ \curlv \dalpha},
\]
namely as a duality pairing in $\Hpara{}$.
The definition \eqref{div-curl} of $\curls$ now yields
\[
	\calI_2 = \para{\f+k^2\Ak \uh}{ \curlv \dalpha}
	= \half{\curls(\f+k^2\Ak \uh)}{ \dalpha}
\]
Since $\dalpha\in H^\frac12(\Gamma)$ and 
$\curls(\f+k^2\Ak \uh)\in L^2(\Gamma)$ because of \eqref{eq:RL2Ak}
and \eqref{assump-f}, we can also write $\calI_2$ as an integral
\begin{equation}
	\label{eq:I2}
	\calI_2 = \int_\Gamma \curls(\f+k^2\Ak \uh) \, \dalpha \ds.
\end{equation}

\step{3}  For the last term $\calI_3$, we integrate by parts according
to \eqref{div-curl}, whence
\[
	\calI_3 = -\ortho{\grad(\Vk \div \uh)}{\dPsiv}.
\]
Since $\div(\RT)\subset L^2(\Gamma)$ we infer that
$\grad(\Vk \div \uh)\in \Lt$ in light of \eqref{eq:RL2Vk}.
This implies that $\calI_3$ is also an integral
\begin{equation}
	\label{eq:I3}
	\calI_3 = -\int_\Gamma \grad(\Vk \div \uh) \cdot \dPsiv\ds.
\end{equation}

\step{4}
Inserting \eqref{eq:I1}-\eqref{eq:I3} back into the sesquilinear form
$a$ yields
\begin{equation}\label{eq:reseq}
	a(\u-\uh,\v) = \int_\Gamma \Res\cdot \dPsiv \ds
	+  \int_\Gamma \cres  \,\dalpha \ds
        \quad\text{for all }\v\in\Hdiv,
\end{equation}
where $\Res\in\Lt$ and $\cres\in L^2(\Gamma)$ are given element-by-element by
\begin{equation}\label{residuals}
\begin{aligned}
	\Res|_\ka &:= \f+k^2\Ak \uh + \grad(\Vk \div \uh) & \text{for all }\ka\in\mesh,\\
	\cres|_\ka &:= \curls(\f+k^2\Ak \uh)& \text{for all }\ka\in\mesh.
\end{aligned}
\end{equation}
We now choose $\alpham=\Ialpha\alpha$ and $\Psivm=\Ipsi \Psiv$
where $\Ialpha$ and $\Ipsi$ are the interpolation operators of
Definitions \ref{def:interpolationVh} and \ref{def:interpolationRT}.
Applying the Cauchy-Schwarz inequality and the 
interpolation estimates \eqref{eq:Interp2} and \eqref{eq:Interp1}
yields
\begin{align*}
	a(\u-\uh,\v) &\le  \|\h^\frac12 \Res\|_{L^2(\Gamma)} \|\h^{-\frac12} \dPsiv\|_{L^2(\Gamma)} + \|\h^\frac12 \cres\|_{L^2(\Gamma)} \|\h^{-\frac12} \dalpha\|_{L^2(\Gamma)}\\
	& \preceq \|\h^\frac12 \Res\|_{L^2(\Gamma)}  \nHortho{}{ \Psiv}+
        \|\h^\frac12 \cres\|_{L^2(\Gamma)} \nH{\frac12}{ \alpha},
\end{align*}
which together with the stability \eqref{eq:StabHD} of the regular
decomposition leads to
\begin{equation*}
a(\u-\uh,\v) \preceq \left(\|\h^\frac12 \Res\|_{L^2(\Gamma)} 
+ \|\h^\frac12 \cres\|_{L^2(\Gamma)}\right) \,\nHdiv{\v}.
\end{equation*}
Combining this with the inf-sup condition \eqref{eq:InfSup} finally implies
\begin{align*}
	\nHdiv{\u-\uh} 
	&\preceq \sup_{\v\in\Hdiv}\frac{a(\u-\uh,\v)}{\nHdiv{\v}} 
	\preceq \|\h^\frac12 \Res\|_{\Gamma} + \|\h^\frac12 \cres\|_{\Gamma}.
\end{align*}
We summarize this derivation in the following theorem.
\begin{theorem}[Upper bound] Let $\f\in\Hpara{}\cap \bm H^0_{{\rm
      curl}}(\Gamma)$, $\u\in\Hdiv$ be the exact solution of
  \eqref{eq:EFIE} and $\uh\in\RT$ be its approximation defined by
  \eqref{eq:Pb}. Then, there exists a constant $C_1>0$ depending on 
shape regularity of $\mesh$ such that the following bound holds
\[
	\nHdiv{\u-\uh}^2 \le C_1 \sum_{\ka\in\mesh} \eta_\mesh^2(\ka)
\]
where the element indicators $\eta_\mesh(\ka)$ are defined as follows in terms of the
residuals $\Res\in\Lt$ and $\cres\in L^2(\Gamma)$ given in \eqref{residuals}
\[
	\eta_\mesh^2(\ka) := h_T\nLt{\Res}{\ka}^2 + h_T\nlt{\cres}{\ka}^2.
\]
\end{theorem}

\begin{remark}[Trace regularity of an incident plane wave]
\rm 
It the right hand side $\f$ is the tangential trace of a plane
wave $\bm E_{inc}$, then we conclude from the analyticity of the plane
wave and of all its derivatives that
\[
	\f=\ptrace(\bm E_{inc}) \in \Hpara{},
\quad
\ptrace(\partial_{x_i} \bm E_{inc}) \in \Hpara{}
\]
for $i=1,2,3$. Therefore, $\f$ satisfies the stated regularity
assumption  \eqref{assump-f}.
\end{remark}

\subsection{Lower Bound}
%
We next show a {\it global} lower bounds for the error indicators $\eta_\mesh^2(\ka)$.
Since $\Res\in\Lt$ and $\cres\in L^2(\Gamma)$ we define the local constants
\[
	\LRes=\int_\ka\Res(\x)\ds(\x)
	\qquad
	\Lcres=\int_\ka\cres(\x)\ds(\x),
        \qquad\text{for all }T\in\mesh,
\]
and their global piecewise constant counterparts 
$\LResG|_\ka = \LRes$ and $\LcresG|_\ka = \Lcres$.

\begin{theorem}[Global lower bound for the residual] Let $\u\in\Hdiv$
  be the exact solution of \eqref{eq:EFIE} and $\uh\in\RT$ be its
  approximation defined by \eqref{eq:Pb}. Then, there exists a
  constant $C_2>0$, only depending on shape regularity of $\mesh$,
  such that the following bound holds
\[
C_2 \nLt{\h^\frac12 \Res}{\Gamma} \le \nHdiv{\u-\uh} + \nLt{\h^\frac12(\Res-\Res_0)}{\Gamma}.
\]
\end{theorem}
\begin{proof}
Let $b_\ka:\Omega\rightarrow \real$ be a bubble function, namely a
Lipschitz function so that
\[
\text{supp } b_\ka\subset\ka, 
\qquad
\int_\ka b_\ka\,d\x=|\ka|\approx\int_\ka b_\ka^2\,d\x,
\]
for a given $T\in\mesh$.
Such a function can be given by a polynomial of degree three on $\ka$ 
consisting of the product of all three barycentric coordinates times 
a real scaling factor.
Let $\Psiv_\ka=\bm\sigma_\ka b_\ka$ with $\bm\sigma_\ka\in\complex^2$
and note that as a direct consequence of the first point we have
	\begin{equation}
		\label{eq:divzeroG}
		\int_\ka \div\Psiv_\ka\ds = \int_{\partial\ka} \Psiv_\ka\cdot \bm n_\ka\de = 0.
	\end{equation}
	For the particular choice $\bm\sigma_\ka=h_\ka\LRes$, we see that 
	\[
		\int_\ka \LRes\cdot\Psiv_\ka\,d\x = h_T\nLt{\LRes}{\ka}^2
	\]	
	and 
	\[
		\nLt{\Psiv_\ka}{\ka}
		\preceq h_\ka\nLt{\LRes}{\ka}
		\preceq \nLt{\Psiv_\ka}{\ka}.
	\]
	We construct a global function $\Psiv$ so that its
          restriction to $T$ coincides with $\Psiv_T$ for all
          $T\in\mesh$. We claim that $\Psiv\in\Hortho{}$ because it is
          made of piecewise polynomials with vanishing 
          normal component on the interelement boundaries of
          $\mesh$. In view of Lemma \ref{lem:HodgeDecomp}, such a
          $\Psiv$ is an admisible test function in \eqref{eq:reseq}
          and, together with the choices $\Psivm=0$ and
          $\alpha=\alpham=0$, yields
	\begin{align*}
		a(\u-\uh,\Psiv)
		&= \int_\Gamma \Res\cdot\Psiv\ds
		= \int_\Gamma (\Res-\LResG)\cdot\Psiv\ds + \int_\Gamma \LResG\cdot\Psiv\ds \\
		& = \int_\Gamma (\Res-\LResG)\cdot\Psiv\ds + \nLt{\h^\frac12\LResG}{\Gamma}^2.
	\end{align*}
	By the continuity of the sesquilinear form $a(\cdot,\cdot)$, we have
\begin{equation}\label{eq:help1G}	
\begin{aligned}
\nLt{\h^\frac12&\LResG}{\Gamma}^2
= a(\u-\uh,\Psiv) - \int_\Gamma (\Res-\LResG)\cdot\Psiv\,d\x\\
&\preceq \nHdiv{\u-\uh}\nHdiv{\Psiv} + \nLt{\h^\frac12(\Res-\LResG)}{\Gamma}\nLt{\h^\frac12\LResG}{\Gamma}.
\end{aligned}
\end{equation}
It remains to estimate $\nHdiv{\Psiv}$.
For $\varphi\in\H{\frac12}$, let $\varphi_0$ denote the elementwise 
average of $\varphi$.
The Bramble-Hilbert Lemma yields
\[
\nLt{\h^{-\frac12}(\varphi-\varphi_0)}{\Gamma} \preceq
|\varphi|_{H^\frac12(\Gamma)},
\]
which in conjunction with \eqref{eq:divzeroG} implies
	\begin{align*}
		\half{\div \Psiv_\ka}{\varphi}
		&= \int_\Gamma \div \Psiv (\varphi-\varphi_0)\,d\x
		\preceq \nlt{\h^\frac12\div \Psiv}{\Gamma}|\varphi|_{H^\frac12(\Gamma)}\\
		&\preceq \nLt{\h^{-\frac12}\Psiv}{\Gamma}|\varphi|_{H^\frac12(\Gamma)}
		\preceq \nLt{\h^{\frac12}\LResG}{\Gamma}|\varphi|_{H^\frac12(\Gamma)}
	\end{align*}
because of the norm equivalence for the discrete function $\Psiv$. Now, by definition
	\[
		\nH{-\frac12}{\div\Psiv} = \sup_{\varphi\in\H{\frac12}}\frac{\half{\div \Psiv}{\varphi}}{|\varphi|_{H^\frac12(\Gamma)}}
		\preceq \nLt{\h^{\frac12}\LResG}{\Gamma},
	\]
	and
	\[
		\nHpara{-}{\Psiv}
		\le \nLt{\Psiv}{\Gamma}
		\preceq \nLt{\h\LResG}{\Gamma}.
	\]
Consequently
	\[
		\nHdiv{\Psiv} 
		\preceq \nLt{\h^{\frac12}\LResG}{\Gamma}
	\]
	which together with \eqref{eq:help1G} implies that 
	\[
	\nLt{\h^\frac12\Res_0}{\Gamma}
	\preceq \nHdiv{\u-\uh} +
        \nLt{\h^\frac12(\Res-\Res_0)}{\Gamma}.
	\]
	Invoking the triangle inequality finally finishes the proof.
\end{proof}

It is important to realize the global nature of the above lower
  bound. This is due to the presence of integral operators $\Vk,\Ak$ in the 
sesquilinear form $a(\cdot,\cdot)$ which lead to a global norm for the
error in \eqref{eq:help1G} regardless of the support of $\Psiv$.

In a very similar fashion, the following theorem can also be proven.

\begin{theorem}[Global lower bound for the curl residual] 
Let $\u\in\Hdiv$ be the exact solution of \eqref{eq:EFIE} 
and $\uh\in\RT$ be its approximation defined by \eqref{eq:Pb}. Then, 
there exists a constant $C_3>0$, only depending on shape regularity of
$\mesh$, such that the following bound holds
\[
C_3\nlt{\h^\frac12 \,\cres}{\Gamma} \le \nHdiv{\u-\uh} + \nlt{\h^\frac12(\cres-\LcresG)}{\Gamma}.
\]
\end{theorem}

\section{Conclusions}
\label{sec:conc}
In this paper we develop 
the first a posteriori error estimates for the electric field 
integral equation on polyhedra. We choose, for simplicity, to derive
residual based error estimates but believe that our theory extends to
other non-residual estimators.
We also choose to develop the theory for polyhedra, the most
interesting and useful case in practice, but we expect the results to
extend to smooth surfaces.
For scattering problems on polyhedra, the solution $\u$ of the
integral equation, or surface current, is not smooth whereas the
regularity of the right-hand side $\f$ is dictated by 
the surface $\Gamma$ because the incident wave is always smooth. This justifies 
our additional regularity assumption \eqref{assump-f} which, coupled 
with the properties $\grad(\Vk \div \uh)\in\L^2(\Gamma),\curls(\Ak
\uh)\in L^2(\Gamma)$, allows us to evaluate the residuals $\Res,\cres$
of \eqref{residuals} in $L^2(\Gamma)$ and thus avoid dealing with
fractional Sobolev norms.
We derive computable
global upper and lower a posteriori bounds for the estimator (up to oscillation terms).
In contrast to PDE, the estimator is global and due to the presence of the
potentials $\Vk,\Ak$ in the definition of the sesquilinear form.
However, the residuals $\Res,\cres$ being evaluated in $L^2(\Gamma)$
can be split elementwise and used to drive an adaptive boundary
element method (ABEM). The actual implementation of ABEM for EFIE is
rather delicate and is not part of the current discussion, which focusses
on the derivation and properties of the estimators.

\bibliographystyle{abbrv}
\bibliography{biblio,bibliospec}
\end{document}

%% file: figs/Gamma.pdf_tex

\begingroup
  \makeatletter
  \providecommand\color[2][]{%
    \errmessage{(Inkscape) Color is used for the text in Inkscape, but the package 'color.sty' is not loaded}
    \renewcommand\color[2][]{}%
  }
  \providecommand\transparent[1]{%
    \errmessage{(Inkscape) Transparency is used (non-zero) for the text in Inkscape, but the package 'transparent.sty' is not loaded}
    \renewcommand\transparent[1]{}%
  }
  \providecommand\rotatebox[2]{#2}
  \ifx\svgwidth\undefined
    \setlength{\unitlength}{421.95541992pt}
  \else
    \setlength{\unitlength}{\svgwidth}
  \fi
  \global\let\svgwidth\undefined
  \makeatother
  \begin{picture}(1,0.83199117)%
    \put(0,0){\includegraphics[width=\unitlength]{Gamma.pdf}}%
    \put(0.11052972,0.69835686){\color[rgb]{0,0,0}\makebox(0,0)[lb]{\smash{$\bm n_j$}}}%
    \put(0.09891996,0.37815349){\color[rgb]{0,0,0}\makebox(0,0)[lb]{\smash{$\Gamma_j$}}}%
    \put(0.61668619,0.23780316){\color[rgb]{0,0,0}\makebox(0,0)[lb]{\smash{$\Gamma_i$}}}%
    \put(0.94876756,0.64093343){\color[rgb]{0,0,0}\makebox(0,0)[lb]{\smash{$\bm n_i$}}}%
    \put(0.62707229,0.79099452){\color[rgb]{0,0,0}\makebox(0,0)[lb]{\smash{$\bm \tau_{ij}$}}}%
    \put(0.73342956,0.67273234){\color[rgb]{0,0,0}\makebox(0,0)[lb]{\smash{$\bm \tau_j$}}}%
    \put(0.71107757,0.49498046){\color[rgb]{0,0,0}\makebox(0,0)[lb]{\smash{$\bm \tau_i$}}}%
    \put(0.41520974,0.46450405){\color[rgb]{0,0,0}\makebox(0,0)[lb]{\smash{$e_{ij}$}}}%
  \end{picture}%
\endgroup

%% file: figs/Triangle.pdf_tex

\begingroup
  \makeatletter
  \providecommand\color[2][]{%
    \errmessage{(Inkscape) Color is used for the text in Inkscape, but the package 'color.sty' is not loaded}
    \renewcommand\color[2][]{}%
  }
  \providecommand\transparent[1]{%
    \errmessage{(Inkscape) Transparency is used (non-zero) for the text in Inkscape, but the package 'transparent.sty' is not loaded}
    \renewcommand\transparent[1]{}%
  }
  \providecommand\rotatebox[2]{#2}
  \ifx\svgwidth\undefined
    \setlength{\unitlength}{356.58039551pt}
  \else
    \setlength{\unitlength}{\svgwidth}
  \fi
  \global\let\svgwidth\undefined
  \makeatother
  \begin{picture}(1,0.45204711)%
    \put(0,0){\includegraphics[width=\unitlength]{Triangle.pdf}}%
    \put(0.53210903,0.37544578){\color[rgb]{0,0,0}\makebox(0,0)[lb]{\smash{$\bm n$}}}%
    \put(0.5470097,0.09336209){\color[rgb]{0,0,0}\makebox(0,0)[lb]{\smash{$\ka_e$}}}%
    \put(0.09663035,0.28579656){\color[rgb]{0,0,0}\makebox(0,0)[lb]{\smash{$\bm \nu_e$}}}%
    \put(0.31199305,0.27843627){\color[rgb]{0,0,0}\makebox(0,0)[lb]{\smash{$e$}}}%
  \end{picture}%
\endgroup